\documentclass[12pt]{amsart}

\usepackage{amsmath}
\usepackage{amsfonts}
\usepackage{amssymb}
\usepackage{amsthm}
\usepackage[T1]{fontenc}

\def\C{\mathbb{C}}

\def\F{\mathbb{F}}

\newcommand{\supp }{\mathrm{supp\,}}

\newcommand{\Jac }{\mathrm{ {\mathcal J}ac}}

\newtheorem{theo}{Theorem}[section]
\newtheorem{definition}[theo]{Definition}
\newtheorem{proposition}[theo]{Proposition}
\newtheorem{corollary}[theo]{Corollary}
\newtheorem{lemma}[theo]{Lemma}
\newtheorem{remark}[theo]{Remark}
\newtheorem{theorem}[theo]{Theorem}
\newtheorem{example}[theo]{Example}

\def\div{\mathsf{D}}
\def\pp{\mathsf{Princ}}
\def\FF{\mathbf{F}}
\def\CC{\mathbf{C}}
\def\bP{\mathsf{P}}

\begin{document}
\title[Existence of  dimension zero divisors]
{On the existence of dimension zero divisors
in algebraic function fields defined over $\F_q$}
\author{S. Ballet}
\address{Institut de Math\'{e}matiques
de Luminy\\ case 930, F13288 Marseille cedex 9\\ France}
\email{ballet@iml.univ-mrs.fr}
\author{C. Ritzenthaler}
\thanks{Second author partially supported by grant MTM2006-11391 from the Spanish MEC}
\address{Institut de Math\'{e}matiques
de Luminy\\ case 930, F13288 Marseille cedex 9\\ France}
\email{ritzenth@iml.univ-mrs.fr}
\author{R. Rolland}
\address{Institut de Math\'{e}matiques
de Luminy\\ case 930, F13288 Marseille cedex 9\\ France}
\email{robert.rolland@acrypta.fr}
\date{\today}
\begin{abstract}
Let $\FF/\F_q$ be an algebraic function field of genus $g$ defined over a finite field $\F_q$.
We obtain new results on the existence, the number and the density of dimension zero divisors of degree $g-k$ in $\FF/\F_q$ where $k$ is an integer $\geq 1$. 
In particular, for $q=2,3$ we prove that there always exists  a dimension zero divisor of degree $\gamma-1$ where $\gamma$ is the $q$-rank of $\FF$. 
We also give a necessary and sufficient condition for the existence of a dimension zero divisor of degree $g-k$ for a hyperelliptic field $\FF$ in terms of its Zeta function.
\end{abstract}

\maketitle

\section{Introduction}
Let $\FF/\F_q$
be an algebraic function field of 
one variable defined over a finite field $\F_q$. We will always 
suppose that the 
full constant field of $\FF/\F_q$ is $\F_q$ and denote by $g$ the genus of $\FF$. 
If $D$ is a (rational) divisor we recall that  the $\F_q$-\emph{Riemann-Roch vector space} associated to $D$ and denoted
${\mathcal L}(D)$ is the following subspace of rational functions
\begin{equation}\label{fct}
{\mathcal L}(D)= \{ x \in \FF ~|~(x)\geq -D\}\cup \{0\}.
\end{equation}
By Riemann-Roch Theorem we know that the dimension of this vector space,
denoted by $\dim (D)$, is related to the genus of $\FF$ and the degree of $D$
by 
\begin{equation} \label{RR}
\dim (D) = \deg (D)-g+1 +\dim (\kappa-D),
\end{equation}
where $\kappa$ denotes a canonical divisor of $\FF/\F_q$ of degree $2g-2$.
In this relation, the complementary term $i(D) =\dim (\kappa-D)$ is called the \emph{index of speciality} and 
is not easy to compute in general.
In particular, a divisor $D$ is non-special when the index of speciality $i(D)$
is zero. Many deep results have been obtained on the study of such divisors in the dual language of curves when the field of definition is algebraically closed. See for instance \cite{arcoal} for a beautiful survey over $\C$. On the contrary, few results are known when the rationality of the divisor is taken into account as in our context where we require the divisor $D$ to be defined over $\F_q$.
We refer to \cite{balb} for known results on the existence of non special divisors of degree $g$ and $g-1$.\\

In the present article, we study a natural generalization of certain results obtained in \textit{loc. cit.} by  looking at \emph{dimension zero divisors}, i.e. divisor $D$ such that $\dim D=0$. Clearly, such a divisor has degree less than or equal to $g-1$. The non special divisors of degree $g-1$ form the borderline case. In Corollary \ref{kmin} we prove that for $g \geq 1$, $q \geq 4$ (resp. q$=3$, resp. $q=2$) and $k \geq 1$ (resp. $k \geq 2$, resp. $k\geq 5$) there always exists a dimension zero divisor of degree $g-k$. When $q \leq 3$, it is not known whether  there exist infinitely many function fields without non special degree $g-1$ divisors (see \cite[Rem.12]{balb} for examples with $g \leq 3$). For $q=p \leq 3$, we slightly clarify the situation by showing that,  when the Jacobian of $\FF/\F_q$ is ordinary,  there always exists a non special divisor of degree $g-1$ (Proposition \ref{ordinary2}). More generally, we show the existence of a dimension zero divisor of degree $\gamma-1$ where $\gamma$ is the so called $p$-rank of $\FF$.   Also, if $\FF/\F_2$ has at least three degree one places, then one can replace $k \geq 5$ by $k \geq 2$ in Corollary \ref{kmin}.\\
If the Zeta function of $\FF$ is known, we give a sufficient condition on its coefficients to have a dimension zero divisor of degree $g-k$ (Theorem \ref{coropi}). In general (Remark \ref{nonsufficient})  the knowledge of the Zeta function is  not sufficient  to distinguish between function fields with or without dimension zero divisor of a certain degree.  However if $\FF$ is hyperelliptic, using results from \cite{pell},  we can give a necessary and sufficient condition (Theorem \ref{hyperelliptic}).\\
Many inequalities and  techniques we used are refinements of the ones developed in \cite{balb}. More specifically most results are derived from  the inequality $A_m < h$ where $A_m$ is the number of effective divisors of a certain degree $m$ and $h$ 
the divisor class number (see Lemma \ref{pp}). To the best of our knowledge, the use of the $p$-rank  and the hyperelliptic case are original. \\
Apart from theoretical interests, our study is motivated by the appearance of such divisors in many applications (see \cite{shtsvl} and \cite{ball5}). Hence we do not only prove existence results but density results as well in order to justify the good behavior of certain algorithms.  \\




Here is an overview of our paper. In Section \ref{prelim}, we recall the basic definitions and notation
for algebraic function fields and elementary and known results on zero dimension divisors.
 Then, in Section \ref{finit} we 
give our main results concerning the existence of
dimension zero divisors. In Section \ref{ex}, we study the special cases $q=2,3$ under various hypotheses.  In Section \ref{density}, we prove that a random draw of a divisor of degree $g-k$
gives with a high probability a dimension zero divisor.
Then we compare the results obtained in section \ref{finit}
with the ones obtained in \cite{tsfa} and \cite{shtsvl} which can be deduced directly
from the known asymptotical properties  of the Zeta-Functions. 

\section{Preliminaries}\label{prelim}

\subsection{Notation.} Let us recall the usual notation (for the basic notions related
to an algebraic function field $\FF/\F_q$ see \cite{stic}). Let $\FF/\F_q$ be a function field of genus $g$.
For any integer $k\geq 1$ we denote by $\bP_k(\FF/\F_q)$ 
the set of places of degree $k$, by $B_k(\FF/\F_q)$ the 
cardinality of this set and by $\bP(\FF/\F_q)=\cup_k \bP_k(\FF/\F_q)$.
The divisor group of $\FF/\F_q$  is denoted by ${\div}(\FF/\F_q)$.
If a divisor $D \in {\div}(\FF/\F_q)$ is such that
$$D= \sum_{P \in \bP(\FF/\F_q)} n_P P,$$
the support of $D$ is the following finite set
$$\supp (D)=\{ P \in \bP(\FF/\F_q)~|~n_P \neq 0\}$$
and its degree is
$$\deg(D)=\sum_{P\in \bP(\FF/\F_q)} n_P\deg(P).$$
We denote by ${\div}_n(\FF/\F_q)$ the set of divisors of degree $n$.
We say that the divisor $D$ is \emph{effective} if for each $P\in \supp(D)$
we have $n_P \geq 0$ and we denote ${\div}_n^+(\FF/\F_q)$ the set of effective divisors of degree $n$ and $A_{n}=\# {\div}_n^+(\FF/\F_q)$.\\
The \emph{dimension} of a divisor $D$, denoted by $\dim(D)$, is the dimension
of the vector space ${\mathcal L}(D)$ defined by the formula (\ref{fct}).\\
Let $x \in \FF/\F_q$, we denote by $(x)$ the divisor
associated to the rational function $x$, namely
$$(x)=\sum_{P\in\bP(\FF/\F_q)} v_P(x)P,$$
where $v_P$ is the valuation at the place $P$.
Such a divisor $(x)$ is called a principal divisor, and
the set of  principal divisors is
a subgroup of ${\div}_0(\FF/\F_q)$ denoted by ${\pp}(\FF/\F_q)$.
The factor group 
$${\mathcal C}(\FF/\F_q)={\div}(\FF_q)/{\pp}(\FF/\F_q)$$
is called the divisor class group. If $D_1$ and $D_2$ are
in the same class, namely if the divisor $D_1-D_2$ is principal,
we will write $D_1 \sim D_2$. We will denote
by $[D]$ the class of the divisor $D$.\\
If $D_1 \sim D_2$, the following holds
$$\deg(D_1)=\deg(D_2), \quad \dim(D_1)=\dim(D_2),$$
so that we can define the degree 
$\deg([D])$ and the dimension $\dim([D])$ of a class.
Since the degree of a principal divisor is $0$, we can define the subgroup
${\mathcal C}(\FF/\F_q)^0$ of classes of degree $0$ divisors in ${\mathcal C}(\FF/\F_q)$.
It is a finite group and we denote by $h$ its order, called the \emph{class number} of $\FF/\F_q$. Moreover if 
$$L(t)=\sum_{i=0}^{2g} a_i t^i=\prod_{i=1}^{g} [(1-\pi_i t)(1-\overline{\pi_i} t)]$$
with $|\pi_i|=\sqrt{q}$ is the numerator of the Zeta function of $\FF/\F_q$, we have  
$h=L(1)$. Finally we will denote $h_{n,k}$ the number of classes of divisors of degree $n$ and of dimension $k$.\\

In the sequel, we may  simultaneously use the dual language of (smooth, absolutely irreducible, projective) curves by associating to $\FF/\F_q$ a unique ($\F_q$-isomorphism class of) curve $\CC/\F_q$ of genus $g$ and conversely to such a curve its function field. Because, by F.K. Schmidt's theorem (cf. \cite[Corollary V.1.11]{stic}) there always exists a rational divisor of degree $1$, the group ${\mathcal C}(\FF/\F_q)^0$ is isomorphic to the group of $\F_q$-rational points on the Jacobian of $\CC$, denoted $\Jac(\CC)$. In particular $h(\FF/\F_q)=\# \Jac(\CC)(\F_q)$. 

\subsection{Elementary results on dimension zero divisors.}
Let $\FF/\F_q$ be a function field of genus $g>0$. We are interested in dimension zero divisors. From Riemann-Roch theorem, they are the divisors $D$ such that $[D]$ does not contain an effective divisor. If $\deg(D)<0$ then $D$ is automatically a dimension zero divisor. In the sequel we then assume that $\deg(D) \geq 0$. Note that the divisor with empty support exists and is effective. So, if $\deg(D)=0$, $D$ is a dimension zero divisor if and only if $D$ is not principal.
Finally, as said in the introduction, since $i(D) \geq 0$, it follows easily from Riemann-Roch theorem that a dimension zero divisor is necessarily of degree less than $g$. The following lemma shows that among them, the degree $g-1$ dimension zero divisors represent the extremal case.
\begin{lemma} \label{induction}
Let $D \in \div_{g-1}(\FF/\F_q)$ be a dimension zero divisor of degree $g-1$. If $B_1(\FF/\F_q)>0$ then for all $k>0$, there exists a dimension zero divisor of degree $g-k$.
\end{lemma}
\begin{proof}
Let $P \in \bP_1(\FF/\F_q)$. Suppose that $D-k P$ is not a dimension zero divisor. Then by Riemann-Roch theorem, there exists a function $x \in \FF$ such that $(x)+D-kP$ is an effective
divisor. But then $(x)+D$ is also effective and $D$ is not a dimension zero divisor. Contradiction.
\end{proof}
Note that if degree $D$ is $g-1$, then $D$ is a dimension zero divisor if and only if $i(D)=0$. Such divisors are particular cases of \emph{non special divisors}. Over an algebraically closed field, it is easy to prove the existence of a non special divisor of degree $g-1$ for any function field $\FF$. However, if we impose the rationality of such a divisor then the question is more subtle and has been studied in \cite{balb} and \cite{ball1}.
Among other results proved there, the following are interesting for our purpose.
\begin{proposition}\label{g}
 If $B_1(\FF/\F _q)\geq  g+1$, then there is a non-special 
divisor such that
$\deg {D} = g-1$ and $\supp(D)\subset \bP_1(\FF/\F _q)$.
\end{proposition}
\begin{remark}
Assume that $D\in \div_g(\FF/\F_q)$ is an effective non-special 
divisor of degree $g\geq 1$. If there exists a degree one 
place $P$ such that $P\not\in \supp(D)$, then 
$D-P$ is a non-special divisor of degree $g-1$.
\end{remark}
Using Lemma \ref{induction}, we know that we get in these cases dimension zero divisors of degree $g-k$ for all $k>0$.
\begin{theorem}  \label{non-special}
Let $\FF/\F_q$ be a function field of genus $g(\FF)\geq 2$. 
If $q\geq 4$  there is
a non-special divisor of degree $g(\FF)-1$.
\end{theorem}
Finally, in order to get rid of the small genus cases, we give the following proposition.
\begin{proposition}
Let $\FF/\F_q$ be a function field of genus $g(\FF) \leq 2$. There is always a dimension zero divisor of degree $0$ except in the following cases :
\begin{enumerate}
\item[\rm(i)]  $g=1$ and $\FF/\F_q$ is given by
$$\begin{cases}
y^2+y=x^5+x^3+1 & q=2, \\
y^2=x^3+2x+2 & q=3,\\
y^2+y=x^3+a & q=4 \; \textrm{with} \; a^2+a+1=0;
\end{cases}$$
\item[\rm(ii)] $g=2, q=2$ and $\FF/\F_2$ is given by
$$y^2+y=x^5+x^3+1 \; \textrm{or}  \quad y^2+(x^3+x+1) y=x^6+x^5+x^4+x^3+x^2+x+1.$$ 
\end{enumerate}
Let $\FF/\F_q$ be a function field of genus $g(\FF) = 2$. There is always a dimension zero divisor of degree $1$ except if $\FF/\F_q=\F_2(x,y)/\F_2$, with\\ $y^2+y=x^5+x^3+1$ or $y^2+y=(x^4+x+1)/x$.
\end{proposition}
\begin{proof}
As said, a degree $0$ divisor is zero dimensional if and only if it is not principal. Hence, a function field  $\FF/\F_q$ has no degree $0$ zero dimension divisor if and only if $h(\FF)=1$. When $g=1$, it is well known that the only cases are the one indicated in the proposition. For $g=2$, this  was proved in \cite{maqu} and \cite{lemaqu}. For the degree $1$ case, this has been studied in \cite[Th.11]{balb}
\end{proof}
\noindent
In the sequel, we assume that $g(\FF)>2$.

\section{On the existence of dimension zero divisors}\label{finit}

In this section we establish our main results on the existence of divisors of degree $g-k$
and dimension zero for an algebraic function field $\FF$ of genus $g$ defined over $\F_q$.

\subsection{Relation with the Jacobian}
First, let us give a key lemma in proving the existence of
divisors of dimension zero. 

\begin{lemma}\label{pp}
Let $n$ be an integer and assume that $A_n=\# {\div}_n^+(\FF/\F_q)  <h$, where $h$ is the class number of $\FF$. Then there exists a divisor
of degree $n$ and dimension zero. 
More precisely, the number of  classes of divisors 
of degree $n$ 
and dimension $zero$, denoted $h_{n,0}$ is greater than or equal to $h-A_n$.
\end{lemma}

\begin{proof}
We know by F.K. Schmidt's theorem (cf. \cite[Corollary V.1.11]{stic})
that there always exists a divisor of degree $1$, say  $D_0 \in {\div}_1(\FF/\F_q)$. For all $n$, we denote by $\psi$  the map from  ${\div}_n(\FF/\F_q)$ into the Jacobian of $\FF/\F_q$ given by
$$\psi(D)=\left [D-n D_0\right ].$$
Since we assumed that $A_n<h$, we have that 
$$\# \psi\left ( {\div}_n^+(\FF/\F_q)\right ) \leq A_n <h,$$
namely the restriction to  ${\div}_n^+(F/\F_q)$ of the map $\psi$ is not surjective. If $[D']$  is a class 
of degree $0$ not in the image of $\psi$, then $[D'+n D_0]$ 
does not contain an effective divisor, hence $D'+n D_0$ 
is a dimension zero divisor of degree $n$.
The number of classes $[D]$ of degree $n$ where $D$ is not 
equivalent to a positive divisor
is $ h_{n,0}\geq h-A_n$.
\end{proof}

\begin{remark}
We can give  examples where a possible  equivalence of the first claim in Lemma \ref{pp} does not hold:  for instance the functions field $\FF_1/\F_2 : y^2 + (x^4 + x^3 + x) y = x^{10} + x^6 + 1$ (resp. $\FF_2/\F_2 : y^2 + (x^4 + x^3 + x^2) y = x^{10} + x^7 + x^3$) is of genus $3$ (resp. $4$) with class number $8$ (resp. $10$) and $A_{2}=9$ (resp. $A_{3}=13$). However $\FF_1$ (resp. $\FF_2$) is ordinary and one can apply Proposition \ref{ordinary} to show that there is a dimension zero divisor of degree $2$ (resp. $3$) in $\div(\FF_1/\F_2)$ (resp. in $\div(\FF_2/\F_2)$). 
\end{remark}

\subsection{Main results}

Some applications (cf. \cite{shtsvl} and \cite{ball5}) require many linearly independent 
divisors of dimension
zero. Then we are not only
interested in the existence of such a divisor,
but also in their number. \\ 
Let $k$ be a positive integer.
Let us set the following notation
\begin{enumerate}
 \item $C_q=
\left \{ 
\begin{array}{ccc}
\frac{2(\sqrt{q}-1)^2}{\sqrt{q}} & \hbox{if} & k \geq 2\\
\frac{(\sqrt{q}-1)^2}{\sqrt{q}} & \hbox{if} & k = 1
\end{array}
\right .
$
 \item $l_q\left(\strut k\right)=C_q\,
q^{\frac{k}{2}}$
 \item $\Delta_q=
\left\{
\begin{array}{ccc}
  q^{\frac{g-k}{2}} & \hbox{if} & k\geq 2\\
  2q^{\frac{g-1}{2}} & \hbox{if} & k = 1.
\end{array}
\right .
$
\end{enumerate}
Then the following result holds:

\begin{theorem}\label{princ}
Assume that $k$ is a  positive integer
satisfying the inequality 
\begin{equation}\label{Cq}
-2\log_q\left ( C_q \right) \leq 
k.
\end{equation}
Then, 
there is at least one dimension zero divisor of degree $g-k$ in $\div_{g-k}(\FF/\F_q)$.
Moreover, the number $h_{n,0}$ of (classes of)  linearly independent divisors of degree
$n=g-k$ and dimension zero satisfies 
\begin{equation}\label{nbz}
h_{n,0} \geq h\left (1-\frac{1}{l_q\left(\strut k \right)}\right )+
\Delta_q. 
\end{equation}
\end{theorem}

\begin{proof}
Note that
$$\log_q\left(\strut l_q\left(\strut k\right)\right)=
\log_q \left(C_q\right)+\frac{k}{2}\geq 0,$$
then $l_q\left(\strut k\right)\geq 1$.
From the functional equation of the zeta function, it can be deduced (see \cite[Lemma 3 (i)]{nixi}) that,  for $g\geq 1$, one has
\begin{equation}
\label{zeta1}
A_n=q^{n+1-g}A_{2g-2-n}+h\frac{q^{n+1-g}-1}{q-1},\, \mbox{ for all }0\leq n\leq 2g-2.
\end{equation}
For $g\geq 2$, it follows from $(\ref{zeta1})$  (see   \cite{lamd} or
 \cite[Lemma 3 and proof of Lemma 6]{nixi}) that
$$\sum_{n= 0}^{g-2}A_nt^n+\sum_{n=  0}^{g-1}q^{g-1-n}A_nt^{2g-2-n}=\frac{L(t)-ht^g}{(1-t)(1-qt)}.$$
Substituting  $t=q^{-1/2}$ in the last identity, we obtain
$$2\sum_{ n= 0}^{g-2}q^{-n/2}A_n+q^{-(g-1)/2}A_{g-1}=\frac{h-q^{g/2}L(q^{-1/2})}
{(q^{1/2}-1)^2q^{(g-1)/2}}.$$
For $i=1 \ldots g$ let us write $\pi_i=q^{1/2} e^{i \theta_i}$. Since $$L(q^{-1/2})=\prod_{i=1}^{g} [(1-\pi_iq^{-1/2})(1-\overline{\pi}_i q^{-1/2})]= 2^{2g} \cdot q^{-g} \cdot  \prod_{i=1}^g \sin^2 \theta_i \geq 0,$$ 
the following inequality holds
\begin{equation}
\label{zeta2}
 2\sum_{ n= 0}^{g-2}q^{(g-1-n)/2}A_n+A_{g-1}\leq \frac{h}{(\sqrt{q}-1)^2}.
 \end{equation}
Note that $A_0=1$ and $A_i \geq 0$ for all $i>0$. Then
from the inequality (\ref{zeta2}), if $k \geq 2$ the following holds
$$A_{g-k} \leq \frac{h}{2q^{\frac{k-1}{2}}(\sqrt{q}-1)^{2}}
-q^{\frac{g-k}{2}}.$$
In the same way  we get
$$A_{g-1} \leq \frac{h}{(\sqrt{q}-1)^{2}}-2q^{\frac{g-1}{2}}.$$
In any case
$$A_{g-k}\leq \frac{h}{l_q\left(\strut k \right )}-
\Delta_q<h.$$
Then there exist
$h_{n,0} \geq h-A_{g-k}$ linearly 
independent divisors  of degree $g-k$ and dimension zero 
and the equation (\ref{nbz}) holds.
\end{proof}

\begin{corollary}\label{kmin}
There exists a zero dimension divisor of degree $g-k$ in $\div_{g-k}(\FF/\F_q)$ as soon as
\begin{itemize}
 \item for $q=2$, $k \geq 5$ and $l_2(5)\approx 1.37$;
 \item for $q=3$,  $k \geq 2$ and $l_3(2)\approx 1.85$;
 \item for $q \geq 4$,  $k \geq 1$.
\end{itemize}
\end{corollary}

Hence, as soon as $q \geq 4$, the situation is optimal. This can be seen as a generalization of Theorem \ref{non-special}.

\begin{remark}
However, the bound on $h_{n,0}$ is not optimal. For instance, we have already noticed that every degree $0$ divisor which is not principal is zero dimensional. Hence $h_{g,0}=h-1$. 
From \cite{lamd}, one knows that for any field $\FF/\F_q$ of genus $g$, $h \geq \lceil q^{g-1} \frac{(q-1)^2}{(q+1)(g+1)} \rceil$. 
A simple computation shows that for any $q$, $h-1 >  h\left (1-\frac{1}{l_q\left(\strut g \right)}\right )+
\Delta_q$ as soon as $g>20$.
\end{remark}
%

If more information is known about the Zeta function of $\FF$, other estimates can be deduced.
In the following Lemma, we give the value of $A_{g-k}$
 in terms of the coefficients of the  polynomial  $L(t)$.
\begin{lemma} \label{intermsai}
\label{A_{g-k}=}
Let $\FF/\F_q$ be a function field of genus $g$ and let $L(t)= \sum_{i=0}^{2g}a_it^i$ be the numerator of its Zeta function.
Then 
$$A_{g-k}=\frac{1 }{ q-1}\left[ q^{-k+1}\left(h-\sum_{i=0}^{g+k-1}a_i\right)-\sum_{i=0}^{g-k}a_i\right].$$ 
\end{lemma}
\begin{proof}

From
$$Z(t)
=\sum_{m=0}^{+\infty}A_mt^m= \frac{L(t)}{(1-t)(1-qt)}
=\frac{ \sum_{i=0}^{2g}a_it^i}{(1-t)(1-qt)}$$
we deduce that for all $m\geq 0$, 
$$A_m=\sum_{i=0}^{m} \frac{q^{m-i+1}-1}{ q-1}a_i.$$
In particular, $$(q-1)A_{g-k}=\sum_{i=0}^{g-k}(q^{g-k-i+1}-1)a_i.$$ 
Since 
$a_i=q^{i-g}a_{2g-i}$, for all $i=0,\ldots g$, we get
$$(q-1)A_{g-k}=q^{g-k+1}\sum_{i=0}^{g-k}q^{-i}a_i-\sum_{i=0}^{g-k}a_i=q^{g-k+1}\sum_{i=0}^{g-k}q^{-i}q^{i-g}a_{2g-i}-\sum_{i=0}^{g-k}a_i.$$ 
Hence $$(q-1)A_{g-k}=q^{-k+1}\sum_{i=0}^{g-k}(a_{2g-i}-a_i)-\sum_{i=0}^{g-k}a_i+q^{-k+1}\sum_{i=0}^{g-k}a_i.$$
Furthermore, we know that $h=L(1)=\sum_{i=0}^{2g}a_i$, therefore 
$$\sum_{i=0}^{g-k}(a_{2g-i}-a_i)=h-\sum_{i=0}^{g+k-1}a_i-\sum_{i=0}^{g-k}a_i,$$
which completes the proof. 
\end{proof}


\begin{theorem}\label{coropi}
Let $k\geq1$ be a fixed integer. 
If $\FF/\F_q$ is an algebraic function field such that 
$q\geq 3$ and $q^{-k+1}\sum_{i=0}^{g+k-1}a_i+\sum_{i=0}^{g-k}a_i\geq0$
(resp. $q=2$ and $q^{-k+1}\sum_{i=0}^{g+k-1}a_i+\sum_{i=0}^{g-k}a_i>0$),
then there exists a dimension zero divisor of degree $n=g-k$.
Moreover, the number $h_{n,0}$ of linearly independent divisors of degree
$n=g-k$ and dimension zero is such that
\begin{equation}\label{nbmax2}
h_{n,0} \geq h-\frac{1 }{ q-1}\left[ q^{-k+1}\left(h-\sum_{i=0}^{g+k-1}a_i\right)-\sum_{i=0}^{g-k}a_i\right].
\end{equation}

\end{theorem}
\begin{proof}
By Lemma  \ref{A_{g-k}=}, we have $A_{g-k}< h$. The result follows using 
Lemma \ref{pp}.
\end{proof}

\begin{example}
We can apply the previous theorem to classical  types of algebraic function fields, in particular a maximal function field $\FF$ over $\F_{q^2}$ 
or its descent over $\F_{q}$ if it exists. In the first case, one has $L(t)=(1+qt)^{2g}$ and
 \begin{equation*}\label{nbmax1} 
h_{n,0} \geq h-\frac{1 }{ q^2-1}\left[ q^{-2k+2}
\left(h-\sum_{i=0}^{g+k-1} 
\left(
\begin{array}{c}
2g \\
i
\end{array}
\right)
q^{i}
\right)-\sum_{i=0}^{g-k}\left(
\begin{array}{c}
2g \\
i
\end{array}
\right)
q^{i}\right].
\end{equation*}
In the latter, $L(t)=(1+qt^2)^g$ and
\begin{equation*}\label{nbmax3}
h_{n,0} \geq h-\frac{1 }{ q-1}\left[ q^{-k+1}
\left(h-\sum_{i=0}^{\lfloor \frac{g+k-1}{2}\rfloor} 
\left(
\begin{array}{c}
g \\
i
\end{array}
\right)
q^{i}
\right)-\sum_{i=0}^{\lfloor \frac{g-k}{2}\rfloor}\left(
\begin{array}{c}
g \\
i
\end{array}
\right)
q^{i}\right].
\end{equation*}
For instance, for the descent of the (maximal) genus $3$ hermitian field over $\F_3$, one gets
$h_{2,0} \geq 42$ and $h_{1,0} \geq 60$ (compare with  $h_{1,0} \geq 12$ obtained with formula (\ref{nbz}) of Theorem  \ref{princ}). 
\end{example}

In the particular case where $\FF/\F_q$ is hyperelliptic, one can give a necessary and sufficient condition in terms of the coefficients of $L$ to obtain a dimension zero divisor of degree $g-k$. 
\begin{theorem} \label{hyperelliptic}
If $\FF/\F_q$ is a hyperelliptic algebraic function field of genus $g>2$, the number $h_{g-k,0}$ of linearly independent divisors of degree $g-k$ and dimension zero is
$$h_{g-k,0}= \sum_{i=g-k+1}^g a_i + \sum_{i=g-k}^{g-1}q^{g-i} a_i
  + (q^k -1) \sum_{i=0}^{g-k-1} q^{g-i-k} a_i.$$
\end{theorem}
\begin{proof}
We use results from \cite[Sec.4]{pell}. Let $h_{n,i}$ be the number of classes of divisors of degree $n$ and of dimension $i \geq 0$ of $\FF/\F_q$. By \cite[Prop.4.3]{pell}, for $0 \leq n \leq 2g-2$ and $i>0$, one has 
$$h_{n,i}=A_{n-2i+2}-(q+1)A_{n-2i} + q A_{n-2i-2}.$$ 
Now for any $n$, $$h=\sum_{i=0}^{\infty} h_{n,i}$$
so $$h_{n,0} = h- \sum_{i=1}^{\infty} h_{n,i}.$$
By the expression of $h_{n,i}$ for $i>0$ above and the fact that $A_i=0$ if $i<0$ we get
$$h_{g-k,0}= h- (A_{g-k}-q A_{g-k-2}).$$
Using the expressions of $A_{g-k}$ from Lemma \ref{intermsai} and 
$$h=\sum_{i=0}^g a_i + \sum_{i=0}^{g-1} q^{g-i} a_i$$
after some simplifications we get the desired equality.
\end{proof}

\begin{remark} \label{nonsufficient}
However, in general, the knowledge of the Zeta function is not sufficient to characterize the existence of a zero dimension divisor. For instance the genus $3$ function field $\FF_1/\F_2 :  y^2 + y = x^7 + x^6 + 1$ has the same $L$-polynomial $L(t)= 8 t^6 - 8 t^5 + 4 t^4 - 2 t^3 + 2 t^2 - 2 t + 1$
as $\FF_2/\F_2 :  y^3+x^2 y^2+ (x^3+1) y=x^4+x^3+1$. Now from \cite[Rem.12]{balb}, $\FF_2$ has no degree $2$ divisor of dimension zero whereas $\FF_1$ (which is hyperelliptic) has by Theorem \ref{hyperelliptic}.
\end{remark}

\section{Particular cases : the cases of $\F_2$ and $\F_3$} \label{ex}
\subsection{Assumptions on the $p$-rank}
Let $\CC/k$ be a genus $g$  (smooth projective absolutely irreducible)  curve over a finite field $k=\F_{p^n}$.  
Classically, one defines the \emph{$p$-rank} $\gamma$ of this curve as the integer $0 \leq \gamma \leq g$ such that $\# \Jac(\CC)[p](\overline{k})=p^{\gamma}$. In particular 
$\CC$ is said to be \emph{ordinary} if $\gamma=g$.  There is another equivalent characterization in terms of the $L$-polynomial, 
namely $\gamma=\deg(L(t) \pmod{p})$ (see \cite{mani}). In particular, $\CC$ is ordinary if and only if $p$ does not divide $a_g$. 
\begin{proposition} \label{ordinary}
Let $\CC$ be an ordinary curve of genus $g>0$ over a finite field $k$ of characteristic $2$.
There is always a  degree $g-1$ zero dimension divisor on $\CC$.
\end{proposition}
\begin{proof}
Let $f \in k(\CC)$ such that $df \ne 0$. Developing $f$ in power series at any point of $\CC$, we see that $df$ has only zeros and poles of even multiplicity. 
Hence there exists a rational divisor of degree $(2g-2)/2=g-1$ such that  $(df)=2 D_0$. It is easy to show that the class of this divisor does not depend on the choice of $f$ 
and it is called the \emph{canonical theta characteristic divisor}. In \cite[Prop.3.1]{stvo}, it is shown that there is a bijection between ${\mathcal L}(D_0)$ and the space of  exact regular 
differentials (i.e. the regular differentials $\omega$ such that $\omega=df$ for $f \in k(\CC)$) . Now by \cite[Prop.8]{serr}, a regular differential $\omega$ is exact if and only 
if ${\mathcal C}(\omega)=0$ where ${\mathcal C}$ is the Cartier operator.  Moreover by  \cite[Prop.10]{serr}, $\Jac(\CC)$ is ordinary if and only if ${\mathcal C}$ is bijective. 
So the only exact regular differential is $0$ and $\dim(D_0)=0$. Hence $D_0$ is the divisor we were looking for.
\end{proof}

\begin{corollary}
Let $\CC/\F_2$ be an ordinary curve of genus $g>0$. Assume that $\# \CC(\F_2)>0$ then for all $k>0$, there exists a dimension zero divisor of degree $g-k$.
\end{corollary}

Note, that the previous proof gives a way to explicitly construct a degree $g-1$ divisor of dimension zero.  We will now generalize  Proposition \ref{ordinary} but without such an explicit construction.

\begin{proposition} \label{ordinary2}
Let $\CC$ be a curve of genus $g>0$ over a finite field $\F_q$ of characteristic $p$ and of $p$-rank $\gamma$.
There is always a  degree $\gamma-1$ zero dimension divisor on $\CC$.
\end{proposition}
\begin{proof}
With the same notation as in the proof of Theorem \ref{hyperelliptic}, for all $g \geq  k>0$ we get
$$h=\sum_{i=0}^{\infty} h_{g-k,i}$$
so $$h_{g-k,0} = h- \sum_{i=1}^{\infty} h_{g-k,i}.$$
Now $$A_{g-k}= \sum_{i=1}^{\infty} \frac{q^i-1}{q-1} h_{g-k,i}$$
hence we can write
$$\sum_{i=1}^{\infty} h_{g-k,i}=  \sum_{i=1}^{\infty} q^i h_{g-k,i}- (q-1) A_{g-k} .$$
Using the expression of $A_{g-k}$ from Lemma \ref{intermsai} and then 
$$h=\sum_{i=0}^{g+k-1} a_i + \sum_{i=g+k}^{2g}  a_i \equiv \sum_{i=0}^{\gamma} a_i \pmod{p}$$
we get for $k=g-\gamma+1$
\begin{eqnarray*}
h_{g-k,0} &=& h (1+q^{-k+1})  - q^{-k+1} \sum_{i=0}^{g+k-1} a_i - \sum_{i=0}^{g-k} a_i - \sum_{i=1}^{\infty} q^i h_{g-k,i} \\
&=&  \sum_{i=0}^{2 g} a_i+q^{-k+1}  \sum_{i=g+k}^{2g} a_i- \sum_{i=0}^{g-k} a_i- \sum_{i=1}^{\infty} q^i h_{g-1,i} \\
&=&  \sum_{i=g-k+1}^{ 2g } a_i+q^{-k+1}  \sum_{i=g+k}^{2g} a_i- \sum_{i=1}^{\infty} q^i h_{g-1,i} \\
& \equiv &  \sum_{i=g-k+1}^{ \gamma } a_i \pmod{p} \\
& \equiv & a_{\gamma} \not \equiv 0 \pmod{p}.
\end{eqnarray*} 
Hence $h_{\gamma-1,0}$ is not zero and hence is positive.
\end{proof}

\begin{remark}

Note that this proposition is interesting only in the case where $q=2$ and  $\gamma=g-k$ with $k \leq 3$ or $q=3$ and  $\gamma=g$.
Indeed, Corollary  \ref{kmin} gives a better result for the other values of $q$ and $k$.

\end{remark}

%

\subsection{Assumption on the number of rational points }
As we have seen in Proposition \ref{g}, if we know that there are many ($>g$) degree one places then there is always a dimension zero divisor of degree $g-k$ for all $k>0$. We want to relax the hypothesis in the case of $q=2$ for $k>1$.
To do so, we need the following lemma which can be found
in \cite{nixi}.

\begin{lemma}\label{lemnumdiveffect}
If $B_1(\FF/\F_q) \geq m\geq1$, then for all $n\geq 2$ one has $$A_n\geq mA_{n-1}-\frac{m(m-1)}{2}A_{n-2}.$$
\end{lemma}

Then the following result holds and improves Proposition \ref{g}.
\begin{theorem}\label{unefois}
If $q=2$, ($g\geq 3$) and $B_1(\FF/\F_2)\geq 3$ then $$A_{g-k}<h(\FF/\F_2)$$ 
for any integer $k\geq2$. Therefore there exists a divisor
of degree $g-k$ and dimension zero for any $k\geq 2$.
\end{theorem}

\begin{proof}
From the inequality (\ref{zeta2}), we obtain 
\begin{equation}
\label{zeta4}
4A_{g-3}+2\sqrt{2}A_{g-2}+A_{g-1}\leq \frac{h}{(\sqrt{2}-1)^2}=(3+2\sqrt{2})h.
\end{equation}
Assume  that $B_1(\FF/\F_2)\geq m=3$, then by Lemma \ref{lemnumdiveffect} applied with $n=g-1$, we have $A_{g-1}+3A_{g-3}\geq 3 A_{g-2}$. 
Hence, using the inequality (\ref{zeta4}), we obtain $$A_{g-3}+(3+2\sqrt{2})A_{g-2}\leq (3+2\sqrt{2})h.$$
Moreover, it is clear that $A_{g-3}\geq 1$ because if $g=3$, $A_{g-3}=A_0=1$ and if $g>3$, $A_{g-3}\geq B_1(\FF/\F_2)=m=3$. 
Hence, we deduce that if $B_1(\FF/\F_2)\geq 3$ and $g\geq 3$, then we have $A_{g-2}<h$. We can  apply Lemma \ref{induction} to get the result.
\end{proof}

\section{Density of dimension zero divisors}\label{density}

\subsection{General result}

In many situations, dimension zero divisors are needed.  For example, the bilinear multiplication algorithm of D.Chudnovsky
and G. Chudnovsky (see \cite{chch}) requires the
random choice of good divisors to set up the algorithmic infrastructure. 
In this context, we draw at random a divisor until
we obtain a divisor having the needed properties.   
From an algorithmic point of view, one can ask what  the expected complexity is to construct the required divisors.  
Until now, it is not at all clear that this construction 
is practical (cf. \cite[Rem. 5]{shtsvl}). 
However, the experiments show that in each particular case,
this random draw quickly gives us  a solution. The following result
explains why this method works.

\begin{proposition}\label{proba}
Let ${\div}_n(\FF/\F_q)$ be the set of divisors of degree $n=g-k$ with $k\geq 1$, provided with the equiprobability distribution. 
If $k\geq -2\log_q(C_q)$ then the probability to draw a degree $g-k$ divisor of dimension zero is greater than or equal to $\left(1-\frac{1}{l_q\left(\strut k\right)}\right)$ 
where $l_q\left(\strut k\right)=C_qq^{\frac{k}{2}}$ .
\end{proposition}

\begin{proof}
 For any integer $n$, the probability to draw a  dimension zero divisor 
among the divisors of degree $n$ is equal to the probability to draw a divisor
 class of dimension zero among the classes of degree $n$.  
We have seen in Theorem \ref{princ} that the number of classes of
linearly independent divisors of degree $n=g-k$ and dimension zero
is 
$$h_{n,0}\geq h\left(1-\frac{1}{l_q\left(\strut k\right)}\right)+\Delta_q>h\left(1-\frac{1}{l_q\left(\strut k\right)}\right).$$
Since the number of classes of degree $n$ is  equal to $h$, we get then result.
\end{proof}

Note that this probability does not depend
upon the value of $g$ and tends to $1$ when $k$ is growing
to infinity. In particular $g$ and $k$ can grow simultaneously
to infinity. For example we can take 
$k =\lfloor \log_q(g) \rfloor$ which satisfies the inequality (\ref{Cq})
and which tends to infinity when $g$ is growing to infinity.

For practical cases, using Proposition \ref{proba}, we see
that it is not necessary to take a very large $k$
to obtain a probability very close to $1$.
For example if $q=16$ and $k=3$ which is
a rather small value, the probability to draw
a divisor of degree $g-3$ and dimension zero 
is $\geq \frac{287}{288}\simeq 0.996$.
If $q=256$ and $k=1$ the probability to draw a non-special divisor
of degree $g-1$ is $\geq \frac{224}{225}\simeq 0.995$. 

\subsection{Comparison with asymptotical results}\label{comp}

In this section,  we compare the previous results with those obtained 
under asymptotical assumptions  on the Zeta Functions
by M. Tsfasman in \cite{tsfa}, M. Tsfasman and S. Vladut in \cite{tsvl} and I. Shparlinski, 
M. Tsfasman and S. Vladut in \cite{shtsvl}. 
First, let us recall 
the notion of asymptotically exact sequence
of algebraic function fields introduced in \cite{tsfa}.


\begin{definition}
Consider a sequence ${\mathcal F}/\F_q=(\FF_k/\F_q)_{k\geq 1}$  
of algebraic function fields $\FF_k/\F_q$ defined over $\F_q$ of genus $g_k$. 
We suppose that  the sequence of genus $g_k$ is an increasing 
sequence growing to infinity. 
The sequence ${\mathcal F}/\F_q$ is called \emph{asymptotically 
exact} if for all $m \geq 1$ the following limit exists:  
$$\beta_m({\mathcal F}/\F_q)= \lim_{g_k \rightarrow \infty} \frac{B_m(\FF_k/\F_q)}{g_k} $$  
where $B_m(\FF_k/\F_q)$ is the number of places of degree $m$
on $\FF_k/\F_q$. 
 \end{definition}
 


Now, let us recall the following two results used by I. Shparlinski, 
M. Tsfasman and S. Vladut in \cite{shtsvl}. These results follow easily from Corollary 2 and Theorem 6 of \cite{tsfa}. 
Note that the proof of Theorem 6  of \textit{loc. cit.} can be found in \cite{tsvl}.

\begin{lemma}\label{lem1}
 Let ${\mathcal F}/\F_q=(\FF_k/\F_q)_{k \geq 1}$ be an asymptotically exact
 sequence 
of algebraic function fields defined over $\F_q$ and 
 $h_k$ be the class number of $\FF_k/\F_q$. 
Then $$log_qh_k=g_k\left( 1+\sum_{m=1}^{\infty}\beta_m \cdot log_q\frac{q^m}{q^m-1} \right)+o(g_k)$$
\end{lemma}

\begin{lemma}\label{lem2}
 Let $A_{a_k}$ be the number of effective divisors of degree $a_k$ on
 $\FF_k/\F_q$.
 If $$a_k\geq g_k\left(\sum_{m=1}^{\infty}\frac{m\beta_m}{q^m-1}
  \right)+o(g_k)$$ 
 then
 $$log_qA_{a_k}=a_k+g_k \cdot \sum_{m=1}^{\infty}\beta_m \cdot log_q\frac{q^m}{q^m-1}+o(g_k).$$
\end{lemma}

These asymptotical properties were established in \cite{tsfa} and \cite{tsvl} in order to estimate the class number $h$ 
of algebraic function fields of genus $g$ defined over $\F_q$ and also in order to estimate their number of classes of effective divisors of degree $m\leq g-1$. Namely, I. Shparlinski, M. Tsfasman and S. Vladut used in \cite{shtsvl} the inequality $2A_{g_k(1-\epsilon)}<h_k$ where $0< \epsilon <\frac{1}{2}$ and $k$ big enough, under the hypothesis of Lemma \ref{lem1}. In the same spirit, we give here a generalization of their result in the following proposition and corollary.
\begin{proposition}\label{doubledivspec}
Let ${\mathcal F}/\F_q=(\FF_k/\F_q)_{k\geq 1}$ be an 
asymptotically exact sequence of algebraic function field defined over $\F_q$.
Let $\epsilon$ and $l$ be two real numbers such that $0< \epsilon <\frac{1}{2}$ and $l\geq 1$.
Then, there exists an integer $k_0$ such that for any integer $k\geq k_0$, we get:  
$$lA_{\lceil g_k(1-\epsilon) \rceil}<h_k$$ 
\end{proposition}

\begin{proof}
The total number of linear equivalence classes of an arbitrary 
degree equals to the divisor class number $h_k$ of $\FF_k/\F_q$, 
which is given by Lemma \ref{lem1}. Moreover, for $g_k$ sufficiently large, 
we have:
$$\sum_{m=1}^{\infty}\frac{m\beta_m}{q^m-1} \leq  \frac{1}{\sqrt{q}+1} \sum_{m=1}^{\infty}\frac{m\beta_m}{q^{\frac{m}{2}}-1}<\frac{1}{2}$$
 since $q\geq 2$ and $ \sum_{m=1}^{\infty} \frac{m\beta_m}{q^{\frac{m}{2}}-1}\leq 1$ 
by Corollary 1 of \cite{tsfa}. As $\epsilon<\frac{1}{2}$, one has
$$\lceil g_k(1-\epsilon) \rceil \geq g_k(1-\epsilon) \geq g_k\left(\sum_{m=1}^{\infty}\frac{m\beta_m}{q^m-1}
  \right)+o(g_k)$$ 
for $k$ big enough.
Therefore, we can apply Lemma \ref{lem2} and compare $\log_q l A_{\lceil g_k(1-\epsilon) \rceil}$ with $\log_q h_k$ given by Lemma \ref{lem1}. Hence, there exists an integer $k_0$ such that for $k \geq k_0$,   
$lA_{\lceil g_k(1-\epsilon) \rceil}<h_k$.
\end{proof}

\begin{corollary}\label{coroasymp}
Let ${\mathcal F}/\F_q=(\FF_k/\F_q)_{k\geq 1}$ be an asymptotically 
exact sequence of algebraic function fields defined over $\F_q$. 
Let $\epsilon$ be
a real number such that $0<\epsilon <\frac{1}{2}$ and $\phi=o(g_k)$ a function such that $g_k(1-\epsilon)+\phi(g_k)$ is an integer.
Then there exists an integer $k_0$ such that for any integer $k\geq k_0$, 
there is a divisor of degree $g_k(1-\epsilon)+\phi(g_k)$ and dimension zero
in ${\div}\left(\FF_k/\F_q\right)$.
\end{corollary}

\begin{proof}
The corollary is a consequence of Proposition \ref{doubledivspec} 
applied with $l=1$
and Lemma \ref{pp}.
\end{proof}

On the other hand Theorem \ref{princ}
implies the following result.
\begin{proposition}
Let $q$ be a prime power and $l \geq 1$ be a real number. Then
for any algebraic function field $\FF/\F_q$ of genus $g$
and any strictly positive integer $k$ such that
$$2\,\log_q(l)-2\,\log_q(C_q)\leq k $$
we have: 
$$A_{g-k}<\frac{h}{l}.$$
\end{proposition}


%
%

If we compare this proposition to Proposition \ref{doubledivspec}, and thus to Tsfasman and Vladut's inequality,
we see that the hypothesis ``asymptotical exact sequence''
is no longer necessary. Moreover, the range 
covered by the divisor order $g-k$ is now larger than 
the one covered by $g(1-\epsilon)$. In particular,
we can now take a constant $k$ or $k$ growing slowly to infinity, for example
$k=\log_q(g)$, which was not possible in Proposition
 \ref{doubledivspec}.

\bibliographystyle{plain}
\bibliography{stdlib_sbrr}

\begin{thebibliography}{10}

\bibitem{arcoal}
Arbarello, Cornalba, Griffiths, and Joe Harris.
\newblock {\em {G}eometry of {A}lgebraic {C}urves}, volume~I.
\newblock Springer, 1985.

\bibitem{ball1}
St\'ephane Ballet.
\newblock Curves with many points and multiplication complexity in any
  extension of ${\F}_q$.
\newblock {\em {F}inite {F}ields and {T}heir {A}pplications}, 5:364--377, 1999.

\bibitem{ball5}
St\'ephane Ballet.
\newblock On the tensor rank of the multiplication in the finite fields.
\newblock {\em {J}ournal of {N}umber {T}heory}, 128:1795--1806, 2008.

\bibitem{balb}
St\'ephane Ballet and Dominique Le~Brigand.
\newblock On the existence of non-special divisors of degree $g$ and $g-1$ in
  algebraic function fields over ${\F}_q$.
\newblock {\em {J}ournal on {N}umber {T}heory}, 116:293--310, 2006.

\bibitem{chch}
David Chudnovsky and Gregory Chudnovsky.
\newblock Algebraic complexities and algebraic curves over finite fields.
\newblock {\em {J}ournal of {C}omplexity}, 4:285--316, 1988.

\bibitem{lemaqu}
JamesLeitzel, Manohar Madan, and Clifford Queen.
\newblock Algebraic function fields with small class number.
\newblock {\em {J}ournal of {N}umber {T}heory}, 7:11--27, 1975.

\bibitem{lamd}
Gilles Lachaud and Mireille Martin-Deschamps.
\newblock Nombre de points des jacobiennes sur un corps finis.
\newblock {\em {A}cta {A}rithmetica}, 56(4):329--340, 1990.

\bibitem{maqu}
Manohar Madan and Clifford Queen.
\newblock Algebraic function fields of class number one.
\newblock {\em {A}cta {A}rithmetica}, 20:423--432, 1972.

\bibitem{mani}
Yuri Manin.
\newblock The hasse-witt matrix of an algebraic curve.
\newblock {\em Izvestiya Akademii Nauk SSSR. Seriya Matematicheskaya},
  25:153--172, 1961.

\bibitem{nixi}
Harald Niederreiter and Chaoping Xing.
\newblock Low-discrepancy sequences and global function fields with many
  rational places.
\newblock {\em {F}inite {F}ields and {T}heir {A}pplications}, 2:241--273, 1996.

\bibitem{pell}
Ruud Pellikaan.
\newblock On special divisors and the two variable zeta function on algebraic
  curves over finite fields.
\newblock In {R}. {P}ellikaan, {M}. {P}erret, and {S}. {V}ladut, editors, {\em
  {A}rithmetic, {G}eometry and {C}oding {T}heory}, pages 175--184, Berlin - New
  York, 1996. Walter de Gruyter.
\newblock Proceedings of AGCT-4 conference, June 28 - July 2, 1993, Luminy.

\bibitem{serr}
Jean-Pierre Serre.
\newblock Sur la topologie des vari\'et\'es alg\'ebriques en caract\'eristique
  $p$.
\newblock {\em Symp. Int. Top. Alg., Mexico City (or \oe uvres {\bf 1},
  501-530)}, pages 24--53, 1958.

\bibitem{shtsvl}
Igor Shparlinski, Michael Tsfasman, and Serguei Vladut.
\newblock Curves with many points and multiplication in finite fields.
\newblock In {H}. {S}tichtenoth and {M}.{A}. {T}sfasman, editors, {\em {C}oding
  {T}heory and {A}lgebraic {G}eometry}, number 1518 in {L}ectures {N}otes in
  {M}athematics, pages 145--169, Berlin, 1992. Springer-Verlag.
\newblock Proceedings of AGCT-3 conference, June 17-21, 1991, Luminy.

\bibitem{stic}
Henning Stichtenoth.
\newblock {\em Algebraic Function Fields and Codes}.
\newblock Number 314 in {L}ectures {N}otes in {M}athematics. Springer-Verlag,
  1993.

\bibitem{stvo}
Karl-Otto St{\"o}r and Jos\'e Voloch.
\newblock A formula for the cartier operator on plane algebraic curves.
\newblock {\em J. f{\"u}r die Reine und Ang. Math.}, 377:49--64, 1987.

\bibitem{tsfa}
Michael Tsfasman.
\newblock Some remarks on the asymptotic number of points.
\newblock In {H}. {S}tichtenoth and {M}.{A}. {T}sfasman, editors, {\em {C}oding
  {T}heory and {A}lgebraic {G}eometry}, volume 1518 of {\em {L}ecture {N}otes
  in {M}athematics}, pages 178--192, Berlin, 1992. Springer-Verlag.
\newblock Proceedings of AGCT-3 conference, June 17-21, 1991, Luminy.

\bibitem{tsvl}
Michael Tsfasman and Serguei Vladut.
\newblock Asymptotic properties of zeta-functions.
\newblock {\em {J}ournal of {M}athematical Sciences}, 84(5):1445--1467, 1997.

\end{thebibliography}
\end{document}